\documentclass{article}
\usepackage[utf8]{inputenc}

\usepackage{amssymb,amsmath,epsfig}
\usepackage{mathtools}
\usepackage{epstopdf}
\usepackage{listings}
\usepackage{xcolor,soul}
\usepackage{bm,xspace}
\usepackage{comment}
\usepackage{graphicx}
\usepackage{url}
\usepackage{float}
\usepackage{pdfpages}
\usepackage{hyperref}
\usepackage{lipsum}
\usepackage{MnSymbol}
\usepackage[shortlabels]{enumitem}
\usepackage[sort&compress,numbers]{natbib}
\usepackage[margin=1in]{geometry}

\usepackage{authblk}

\newcommand{\ber}{\begin{eqnarray}}
\newcommand{\eer}{\end{eqnarray}}

\newtheorem{remark}{\noindent Remark}
\newcommand{\be}{\begin{equation}}
\newcommand{\ee}{\end{equation}}
\newcommand{\bal}{\begin{align}}
\newcommand{\eal}{\end{align}}
\newcommand{\balnonum}{\begin{align*}}
\newcommand{\ealnonum}{\end{align*}}

\def\qed{\hfill \vrule height1.3ex width1.2ex depth-0.1ex}

\newtheorem{problem}{\noindent Problem}

\newtheorem{theorem}{\noindent Theorem}

\newtheorem{lemma}{\noindent Lemma}

\newtheorem{conjecture}{\noindent Conjecture}

\newenvironment{proof}{\noindent{\bf Proof:}}

\newcommand{\indicatorpar}[1]{\textbf{1}_{\{#1\}}}
\newcommand{\indicator}[1]{\textbf{1}_{#1}}

\title{A correlation inequality for random points in a hypercube with some implications}
\date{\today}
\author{Royi Jacobovic and Or Zuk}
\affil{Department of Statistics and Data Science \\ The Hebrew University of Jerusalem \\
Mt. Scopus, 9190501, Jerusalem, Israel \\
Email: \href{royi.jacobovic@mail.huji.ac.il}{{\it royi.jacobovic@mail.huji.ac.il}}, \href{or.zuk@mail.huji.ac.il}{{\it or.zuk@mail.huji.ac.il}}}

\begin{document}
\maketitle

\abstract{Let $\prec$ be the product order on $\mathbb{R}^k$ and assume that $X_1,X_2,\ldots,X_n$ ($n\geq3$) are i.i.d. random vectors distributed uniformly in the unit hypercube $[0,1]^k$. Let $S$ be the (random) set of vectors in $\mathbb{R}^k$ that $\prec$-dominate all vectors in $\{X_3,..,X_n\}$, and let $W$ be the set of vectors that are not $\prec$-dominated by any vector in $\{X_3,..,X_n\}$. 
The main result of this work is the correlation inequality
\begin{equation*}
    P(X_2\in W|X_1\in W)\leq P(X_2\in W|X_1\in S)\,.
\end{equation*}
For every $1\leq i \leq n$ let $E_{i,n}$ be the event that $X_i$ is not $\prec$-dominated by any of the other vectors in $\{X_1,\ldots,X_n\}$. 
The main inequality yields an elementary proof for the result that the events $E_{1,n}$ and $E_{2,n}$ are asymptotically independent as $n\to\infty$. Furthermore, we derive a related combinatorial formula for the variance of the sum $\sum_{i=1}^n \indicator{E_{i,n}}$, i.e. the number of maxima under the product order $\prec$, and show that certain linear functionals of partial sums of $\{\indicator{E_{i,n}};1\leq i\leq n\}$ are asymptotically normal as $n\to\infty$.}

\section{Introduction}
For every positive integer $l \in \mathbb{N}$, denote $[l]\equiv\{1,2,\ldots.l\}$ and let $\prec$ be the product order on $\mathbb{R}^k$, \textit{i.e.,} for any two real vectors $x=(x_1,x_2,\ldots,x_k)$ and $y=(y_1,y_2,\ldots,y_k)$, 
\begin{equation}\label{eq: product order}
 x\prec y \Leftrightarrow \left( x_i\leq y_i\ ,\ \forall i\in[k]\right)\,.
\end{equation}
In addition, $\not\prec$ is the negation of $\prec$. 

Let $X_1,X_2,\ldots,X_n$ ($n\geq3$) be i.i.d. random vectors distributed uniformly on $[0,1]^ k$ and define two random sets:
\begin{equation}
 W\equiv\left\{x\in\mathbb{R}^k; x\not\prec X_i,\forall 3\leq i\leq n\right\}\ \ , \ \ S\equiv\left\{x\in\mathbb{R}^k; X_i\prec x,\forall 3\leq i\leq n\right\}
\end{equation}
The main result of this work is the following correlation inequality
\be
\label{eq:main inequality}
 P(X_2\in W|X_1\in W)\leq P(X_2\in W|X_1\in S)\,.
\ee
Observe that $S\subset W$ with probability one, hence the conditioning on $X_1$ on the right-hand side above is more restrictive than the conditioning on the left-hand side. 

For each $i\in [n]$ the vector $X_i$ is said to be a maximum of $\textbf{X}_n\equiv\{X_1,X_2,\ldots,X_n\}$ if and only if no \textit{other} vector in $\textbf{X}_n$ $\prec$-dominates $X_i$. Since the product order $\prec$ is a partial order, the set of maxima of $\textbf{X}_n$ might include more than one element. Define the events
\be
 E_{i,n}\equiv\{X_i\text{ is a maximum of $\textbf{X}_n$}\}\ \ , \ \ 1\leq i\leq n\,. 
\ee
Then, the inequality \eqref{eq:main inequality} is used in order to provide an elementary proof for the result that $E_{1,n}$ and $E_{2,n}$ are asymptotically independent as $n\to\infty$. This asymptotic independence result is applied in order to derive a central limit theorem (CLT) for certain linear functionals of partial sums of $\{\indicator{E_{i,n}};i\in[n]\}$.

The rest of this work is organized as follows: Section \ref{sec: correlation} is about the correlation inequality \eqref{eq:main inequality}. In particular, it includes the proof as well as a discussion about a conjectured extension of \eqref{eq:main inequality}. Section \ref{sec:related results} contains the asymptotic independence and the CLT statements along with their proofs.

\subsection{Related literature}
\label{sec:literature}
Correlation inequalities are inequalities satisfied by a correlation function in a probabilistic model. Several well-known correlation inequalities arise in diverse fields of probability theory: In models originating from statistical mechanics \cite{griffiths1967correlations,ginibre1970general} and percolation theory \cite{van1985inequalities}. For the Gaussian case, an inequality was suggested in \cite{pitt1977gaussian}, with some applications, \textit{e.g.,} \cite{li1999gaussian,shao2003gaussian}. 
Additional correlation inequalities appear in \cite{karlin1980classes1,karlin1980classes2,karlin1981total,rinott1993correlation}. 
For example, consider the FKG inequality \cite{fortuin1971correlation} which has several generalizations and applications in probability, \textit{e.g.,} \cite{barbato2005fkg,fishburn1988match,last2011poisson,shepp1982xyz,tassion2016crossing,teixeira2009interlacement,van1992lru}. 

Of particular relevance to the current work are correlation inequalities related to linear extensions of finite partially ordered sets. In this settings, there is a set with $\ell$ elements $\{U_1,U_2,\ldots,U_\ell\}$ which is equipped with a partial order $\prec_1$. With a bit abuse of notation, we may represent $\prec_1$ as the set of all $(i,j)\in[\ell]\times[\ell]$ such that $U_i\prec_1U_j$. Let $\prec_2$ be a random linear extension of $\prec_1$ which is distributed uniformly over the set of all linear extensions of $\prec_1$. In this probability space, for every $I\subseteq[\ell]\times[\ell]$, define an ordering event $E_I\equiv\{U_i\prec_2 U_j,\forall(i,j)\in I\}$. 
Several inequalities were proposed for co-occurrence of ordering events, i.e., $P(E_A | E_B) \geq P(E_A)$ for some choices of $A,B\subseteq[\ell]\times[\ell]$.
Equivalently, we can assume $n$ i.i.d. continuous uniform random variables $U_1,U_2,\ldots,U_n \sim U(0,1)$. Denote for every $I\subseteq[\ell]\times[\ell]$ the event $\widetilde{E}_I\equiv\left\{U_i<U_j,\forall(i,j)\in I\right\}$. Then, the inequality is expressed as $P(\widetilde{E}_A | \widetilde{E}_B\cap \widetilde{E}_{\prec_1}) \geq P(\widetilde{E}_A | \widetilde{E}_{\prec_1})$.
For example, consider the XYZ inequality \cite{shepp1982xyz,fishburn1984correlational}, which is related to the FKG inequality and can be phrased as follows:
\begin{theorem} (XYZ-inequality)
	\label{thm:XYZ}
	Let $\prec_1$ be a subset of $[\ell]\times[\ell]$ which is representing a partial order. Then: 
	\be
	P(U_1 < U_2 | \widetilde{E}_{\prec_1}) \leq P(U_1 < U_2 | \widetilde{E}_{\prec_1}\cap\{U_1 < U_3\}).
	\label{eq:XYZ}
	\ee
\end{theorem}

While Eq. \eqref{eq:XYZ} seems intuitive, several natural generalizations of it do not hold in general (see \cite{shepp1982xyz,graham1983applications,shepp1980fkg}). 
For example, taking $\prec_1 \equiv \{ (1,3) \}$ and replacing $B=\{(1,3)\}$ by $B=\{(1,3),(3,4)\}$ leads to a false statement. 
Other inequalities were proposed in \cite{graham1980some,shepp1980fkg}, where $[\ell]$ is partitioned into two disjoint subsets $N_1,N_2$ such that $A$ and $B$ contain only pairs $(i,j)\in N_1\times N_2$, and either the restriction of $\prec_1$ to each of $N_1$ and $N_2$ is a (full) linear order, or $\prec = \prec_{1,1} \bigcup \prec_{1,2}$ such that for each $i=1,2$, $\prec_{1,i}$ is a partial order of $N_i$.

Additional details on related inequalities and counter-examples are available in \cite{correlationinequalities}. Another intriguing result is that although Theorem \ref{thm:XYZ} implies that $P(U_1 < U_2 < U_3) \leq P(U_1 < U_2) P(U_2 < U_3)$, it was shown in \cite{brightwell1985universal} that it may not hold that $P(U_1 < U_2 < U_3 < U_4) \leq P(U_1 < U_2 < U_3) P(U_3 < U_4)$. 
In the above inequalities, $A$ and $B$ are themselves subsets of $[\ell]\times[\ell]$ which are representing partial orders. The main result of the current work \eqref{eq:main inequality} goes beyond such inequalities, as it refers to more complex ordering events that cannot be represented as partial order events, but only as intersections of unions of pairwise ordering events. 

Another strand of relevant literature is about the number of maxima in a model with $n$ random points distributed uniformly in the unit hypercube $[0,1]^k$. This model is related to several problems in analysis of linear programming \cite{blair1986random} and of maxima-finding algorithms \cite{chen2012maxima, devroye1999note,dyer1998dominance,golin1994provably,tsai2003efficient}. Furthermore, it also appears in the context of game theory \cite{o1981number} and in the analysis of random forest algorithms \cite{biau2016random,scornet2015consistency}. 
Theoretical results regarding this model appear in, \textit{e.g.,} \cite{bai1998variance,bai2005maxima,barbour2001number,barndorff1966distribution,baryshnikov2000supporting}. In particular, as to be explained in Section \ref{sec:asymptotic independence}, the asymptotic independence of $E_{1,n}$ and $E_{2,n}$ as $n\to\infty$ might be derived as a consequence of the results in this literature. Remarkably, \cite{bai2005maxima} includes a result regarding asymptotic normality with the corresponding Berry-Essen bound of the number of maxima $\sum\limits_{i=1}^n \indicator{E_{i,n}}$. In the current paper, we derive different asymptotic normality results, for certain linear functionals of partial sums of $\{\indicator{E_{i,n}};i\in[n]\}$. The derivation will follow via an application of some limit theorems for triangular arrays of row-wise exchangeable random variables \cite{weber1980martingale}.

For simplicity, we present our results for the case where all coordinates have a continuous uniform marginal distribution. Since all events of interest can be expressed in terms of the ranks of the corresponding random variables, the major results of the current work still hold when the marginal $U(0,1)$ distribution is replaced by any continuous distribution. For a recent study about the effect of discrete marginal distributions implying non-zero probability for ties, see \cite{jacobovic2021phase}. 

\section{A correlation inequality}\label{sec: correlation}
For every $i, j\in [n]$ denote $B_{ij} \equiv\{ X_i \prec X_j \}$ and $\overline{B_{ij}}\equiv\{ X_i \not\prec X_j \}$. Then, Eq. \eqref{eq:main inequality} can be rephrased as follows:
\begin{theorem} \label{thm:dominate_conditioning_inequality}
\be
P\left(\bigcap_{j=3}^n \overline{B_{2j}} \bigg| \bigcap_{j=3}^n \overline{B_{1j}}\right) \leq 
P\left(\bigcap_{j=3}^n \overline{B_{2j}} \bigg| \bigcap_{j=3}^n B_{j1}\right)\,. 
\label{eq:dominate_conditioning_inequality}
\ee
\end{theorem}

\begin{proof}
Let $x_1, x_2 \in \mathbb[0,1]^k$ be the values of $X_1,X_2$ respectively, and denote $x_{[2]}\equiv[x_1,x_2]$ such that for every $i=1,2$, $x_i=(x_{i1},x_{i2},\ldots,x_{ik})$. It will be shown that for every fixed $x_1,x_2 \in [0,1]^k$ 
\be \label{eq:conditioning_x12}
P(\bigcap_{j=3}^n \{ \overline{B_{2j}} \cap \overline{B_{1j}} \} | x_{[2]}) P(\bigcap_{j=3}^n B_{j1} | x_{[2]}) \leq 
P(\bigcap_{j=3}^n \{ \overline{B_{2j}} \cap B_{j1} \} | x_{[2]})
P(\bigcap_{j=3}^n \overline{B_{1j}} | x_{[2]}) 
\ee
and the theorem will follow by integration over $x_{[2]}$. Due to the conditioning on $x_{[2]}$, we can write the above inequality as
\be
\prod_{j=3}^n P( \{ \overline{B_{2j}} \cap \overline{B_{1j}} \} | x_{[2]}) P( B_{j1} | x_{[2]}) \leq 
\prod_{j=3}^n P( \{ \overline{B_{2j}} \cap B_{j1} \} | x_{[2]})
P(\overline{B_{1j}} | x_{[2]}) 
\ee
and by symmetry 
\be
\Big[P( \{ \overline{B_{23}} \cap \overline{B_{13}} \} | x_{[2]}) P( B_{31} | x_{[2]})\Big]^{n-2} \leq 
\Big[P( \{ \overline{B_{23}} \cap B_{31} \} | x_{[2]})
P(\overline{B_{13}} | x_{[2]}) \Big]^{n-2}.
\ee
Hence, it is enough to prove: 
\be
P( \{ \overline{B_{23}} \cap \overline{B_{13}} \} | x_{[2]}) P( B_{31} | x_{[2]}) \leq
P( \{ \overline{B_{23}} \cap B_{31} \} | x_{[2]})
P(\overline{B_{13}} | x_{[2]}).
\label{eq:inequality_one_vec}
\ee
In order to compute the probabilities of each of the events above, observe that: 
\begin{align}
P(B_{31} | x_{[2]}) &= \prod_{i=1}^k x_{1i} \: , \\
P(\overline{B_{13}} | x_{[2]}) &= 1 - \prod_{i=1}^k (1-x_{1i}) \: , \nonumber \\
P( \overline{B_{23}} \cap \overline{B_{13}} | x_{[2]}) &= 1 - P( B_{23} \cup B_{13} | x_{[2]}) \nonumber \\
&= 1 - P( B_{23}| x_{[2]}) - P( B_{13}| x_{[2]}) + P( B_{23} \cap B_{13} | x_{[2]}) \nonumber \\
&= 1 - \prod_{i=1}^k (1-x_{1i}) - \prod_{i=1}^k (1-x_{2i}) + \prod_{i=1}^k (1-\max\{x_{1i},x_{2i}\}) \: , \nonumber \\
P( \overline{B_{23}} \cap B_{31} | x_{[2]}) &= P( B_{31} | x_{[2]}) - P( B_{23} \cap B_{31} | x_{[2]}) \nonumber \\
&= \prod_{i=1}^k x_{1i} - \indicatorpar{x_2\prec x_1} \prod_{i=1}^k (x_{1i}-x_{2i}).\nonumber
\label{eq:three_boxes_volumes}
\end{align}
Thus, an insertion of the probabilities from Eq. \eqref{eq:three_boxes_volumes} into Eq. \eqref{eq:inequality_one_vec} yields that it is remained to prove: 
\begin{align}
& \Big[ \prod_{i=1}^k x_{1i} - \indicatorpar{x_2 \prec x_1} \prod_{i=1}^k (x_{1i}-x_{2i}) \Big] \Big[1 - \prod_{i=1}^k (1-x_{1i}) \Big] \geq \nonumber \\ 
&\Big[ 1 - \prod_{i=1}^k (1-x_{1i}) - \prod_{i=1}^k (1-x_{2i}) + \prod_{i=1}^k (1-\max\{x_{1i}, x_{2i}\}) \Big] \prod_{i=1}^k x_{1i}.
\end{align}
Then, subtraction of the common terms from both sides implies that the following inequality has to be proven: 
\be
\indicatorpar{x_2\prec x_1} \prod_{i=1}^k (x_{1i}-x_{2i}) \Big[1 - \prod_{i=1}^k (1-x_{1i}) \Big] \leq \Big[\prod_{i=1}^k (1-x_{2i}) - \prod_{i=1}^k \big(1-\max \{x_{1i} ,x_{2i}\}\big) \Big] \prod_{i=1}^k x_{1i} .
\ee
When $ x_2\not\prec x_1$, the left-hand side is zero, and the inequality holds because the right-hand side is always non-negative. When $x_2\prec x_1$, it is left to show that:
\be
\label{eq:G_volume_x1_x2}
G(x_2 ; x_1) \equiv \Big[\prod_{i=1}^k (1-x_{2i}) - \prod_{i=1}^k (1-x_{1i}) \Big] \prod_{i=1}^k x_{1i} - \prod_{i=1}^k (x_{1i}-x_{2i}) \Big[1 - \prod_{i=1}^k (1-x_{1i}) \Big] \geq 0. 
\ee
In particular, $x_2\prec x_1$ implies that $x_2 \in \bigtimes\limits_{i=1}^k [0, x_{1i}]$. We proceed by first showing that $G(x_2 ; x_1) \geq 0$ on the boundary of the box $\bigtimes\limits_{i=1}^k [0, x_{1i}]$, and will later show that 
$G(x_2 ; x_1) \geq 0$ also in the interior. 
First, if $x_{2i} = x_{1i}$ for some $i\in [k]$, then
\be
G(x_2 ; x_1) = \Big[\prod_{i=1}^k (1-x_{2i}) - \prod_{i=1}^k (1-x_{1i}) \Big] \prod_{i=1}^k x_{1i} \geq 0
.
\label{eq:boundary_one}
\ee
Next, if $\textbf{0}_k$ is the zero-vector in $\mathbb{R}^k$, then 
\be
G(\bm{0_k} ; x_1) = \Big[1 - \prod_{i=1}^k (1-x_{1i}) \Big] \prod_{i=1}^k x_{1i} - \prod_{i=1}^k x_{1i} \Big[1 - \prod_{i=1}^k (1-x_{1i}) \Big] = 0.
\label{eq:boundary_zero}
\ee
For every $i \in [k]$ let $x_{2 [-i]} \equiv (x_{21},..,x_{2 (i-1)}, x_{2 (i+1)}, ..,x_{2 k})$ and define the real-valued function 
\be
G_{x_{2 [-i]}, x_1}(y) \equiv G\big((x_{21},..,x_{2 i-1}, y, x_{2 i+1}, ..,x_{2 k}) ; x_1\big)\ \ , \ \ y\in\mathbb{R}.
\label{eq:G_x_def}
\ee
Then, Eq. \eqref{eq:G_volume_x1_x2} implies that $G_{x_{2 [-i]},x_1}(\cdot)$ is a linear function of $y$. Therefore, its minimum in a closed interval is obtained at the boundary, \textit{i.e.}, for every $x_1 \in [0,1]^k$, $x_{2 [-i]} \in \bigtimes\limits_{j=1, j \neq i}^k [0, x_{1j}]$ and $y \in [0, x_{1i}]$ deduce that
\be
G_{x_{2 [-i]},x_1}(y) \geq \min\left\{G_{x_{2 [-i]},x_1}(0), G_{x_{2 [-i]},x_1}(x_{1i})\right\}.
\ee
In particular, Eqs. \eqref{eq:boundary_one} and \eqref{eq:boundary_zero} yield that 
\be
G_{\bm{0_{k-1}},x_1}(y) \geq \min \big(G_{\bm{0_{k-1}},x_1}(0), G_{\bm{0_{k-1}},x_1}(x_{1i}) \big) = 0
\ee
for every $y\in[0, x_{1i}]$. Thus, $G(x_2 ; x_1) \geq 0$ for every vector $x_2 \in \bigtimes\limits_{i=1}^k [0, x_{1i}]$ with at least $k-1$ zero coordinates. We next continue by induction. Let $x_{2 [-i]} \in \bigtimes\limits_{i=1}^k [0, x_{1i}]$ be a vector with $k-2$ zero coordinates. Then, since $G_{x_{2 [-i]},x_1}(0)$ is obtained by evaluating $G$ at a vector with $k-1$ zero coordinates, deduce that
\be
G_{x_{2 [-i]},x_1}(y) \geq \min \big(G_{x_{2 [-i]},x_1}(0), G_{x_{2 [-i]},x_1}(x_{1i}) \big) \geq 0
\ee
for every $y\in[0, x_{1i}]$. Proceeding by induction, we get that $G_{x_{2 [-i]},x_1}(x) \geq 0$ whenever $x_{2i} = 0$ for some $i\in [k]$.
As a result deduce that $G(x_2 ; x_1) \geq 0$ for any $x_2$ lying at the boundary of the box $\bigtimes\limits_{i=1}^k [0, x_{1i}]$. Next, for any $x_{2 [-i]} \in \bigtimes\limits_{j=1, j \neq i}^k [0, x_{1j}]$ and any $y \in [0, x_{1i}]$ we have: 
\be
G_{x_{2 [-i]},x_1}(y) \geq \min \big(G_{x_{2 [-i]},x_1}(0), G_{x_{2 [-i]},x_1}(x_{1i}) \big) \geq 0.
\ee
Therefore, $G(x_2 ; x_1) \geq 0$ for every $x_2 \in \bigtimes\limits_{i=1}^k [0, x_{1i}]$ which completes the proof. \qed \newline\newline
\end{proof}

\subsection{Conjecture} \label{sec:conjecture}
It is natural to consider generalizations the main result presented in Theorem \ref{thm:dominate_conditioning_inequality}. 
This theorem and several other correlation inequalities can be represented as special cases of the following settings: 

\begin{problem}
	\label{conj:DNF_CNF} 
	Consider $n$ i.i.d. uniform random variables $U_1,U_2,\ldots,U_\ell \sim U(0,1)$. For every $(i,j)\in[\ell]\times[\ell]$ denote $E_{ij} \equiv \{U_i < U_j\}$. Find the most general conditions on $T,S\geq 1$,
	$F_1,F_2..,F_T, H_1,H_2 .., H_S \subset [\ell] \times [\ell]$ under which the following inequality is correct: 
	\be
	P( \bigcap_{t=1}^T \bigcup_{(i,j) \in F_t} E_{ij} | 
	\bigcap_{s=1}^S \bigcup_{(i,j) \in H_s} E_{ij} ) \leq 
	P(\bigcap_{t=1}^T \bigcup_{(i,j) \in F_t} E_{ij} | 
	\bigcap_{s=1}^S \bigcap_{(i,j) \in H_s} E_{ij})\,. 
	\label{eq:CNF_DNF_inequality}
	\ee
\end{problem}

\begin{remark}\label{remark XYZ}\normalfont
	The inequality in Problem \ref{conj:DNF_CNF} is a generalization of the XYZ-inequality. 
For a partial order with $q$ pairwise order relations $\prec = \{(a_i, b_i) , \: i\in[q] \}$ take
	$S = q + 1$ with $H_i = \{(a_i, b_j)\}$ for $i\in[q]$ and $H_S = \{ (a_1, b_1), (1, 3) \}$. Take also $T=1, F_1 = \{(1,2)\}$.	Then Eq. \eqref{eq:CNF_DNF_inequality} reduces to Eq. \eqref{eq:XYZ} with $\ell=n$.
\end{remark}

\begin{remark}\label{remark: correlation}\normalfont
	The inequality \eqref{eq:CNF_DNF_inequality} is also a generalization of Theorem \ref{thm:dominate_conditioning_inequality}. Let $\ell = n \cdot k$ and for each $i\in[n]$ define $X_i\equiv(U_{(i-1)k+1},U_{(i-1)k+2},\ldots, U_{ik})$. Then, 
	\begin{equation}
	 B_{ij}=\left\{X_i\prec X_j\right\}= \bigcap\limits_{r=1}^k \{U_{(i-1)k+r} \leq U_{(j-1)k+r}\}\ \ , \ \ \forall (i,j)\in[n]\times[n]\,,
	\end{equation}\textit{i.e.,} \eqref{eq:dominate_conditioning_inequality} is a special case of \eqref{eq:CNF_DNF_inequality}
	with $T=S=n-2$ and
	\begin{align}
	 & F_t = \big\{ \left(k+1, (t+1)k+1\right),(k+2, (t+1)k+2), \ldots, (2k, (t+2)k) \big\}\ \ , \ \ \forall t\in[n-2]\,, \nonumber \\
	 & H_s = \big\{ \left(1, (s+1)k+1\right),\left(2, (s+1)k+2\right),\ldots, \left(k, (s+2)k\right) \big\}\ \ , \ \ \forall s\in[n-2]\,.
	\end{align}
	
\end{remark}

\begin{remark}\label{remark: counterexample}\normalfont
Several of the counter-examples discussed in Section \ref{sec:literature} and the references therein can be phrased as violating the inequality \eqref{eq:CNF_DNF_inequality}. For example, take 
$n=4$ with $T=S=2$ and the following sets: 
	\begin{align}
		F_1 &= \{(1,4)\}\:; \: F_2=\{(1,4), (2,4)\}, \nonumber \\
		H_1 &= \{(1,3)\}\:; \: H_2 = \{(1,3), (2,3)\}.
	\end{align}
	Then the inequality \eqref{eq:CNF_DNF_inequality} reduces to (for some similar related examples, see \cite{shepp1982xyz}): 
	\be
	P(U_1 < U_4 | U_1 < U_3) \leq P ( U_1 < U_4 | U_1, U_2 < U_3).
	\ee
	
	By symmetry, there are $12$ permutations (out of $24$) over $U_1,U_2,U_3,U_4$ satisfying $U_1<U_3$,and $8$ permutations satisfying $U_1 < U_3, U_4$, and similarly $8$ permutations satisfying $U_1, U_2 < U_3$. Therefore, the left hand side of the above inequality is $\frac{8}{12}$. To get the numerator of the right hand side, we need to count the permutations satisfying $U_1, U_2 < U_3$ and $U_1< U_4$. There are five such permutations: 
$\{ (1, 2, 3, 4), (1, 2, 4, 3), (2, 1, 3, 4), (2, 1, 4, 3), (1, 4, 2, 3) \}$,
	hence the right hand side is $\frac{5}{8} < \frac{8}{12}$ and the inequality fails.
\end{remark}
Remark \ref{remark: counterexample} implies that the inequality \eqref{eq:CNF_DNF_inequality} is not true in general. However, Remark \ref{remark XYZ} and Remark \ref{remark: correlation} imply that the inequality \eqref{eq:CNF_DNF_inequality} may hold under certain restrictions on the sets $F_t, H_s$ and their relationships. This motivates finding the most general conditions under which the inequality \eqref{eq:CNF_DNF_inequality} is correct. We suspect that Theorem \ref{thm:dominate_conditioning_inequality} might be extended to the following more general case:

\begin{conjecture}
	Assume that $U_1,..,U_k, U_{k+1},..,U_{2k}, U_{2k+1},...,U_{n} \sim U(0,1)$ are i.i.d. uniform random variables. Let $F_1,..,F_T \subset [k] \times \{2k+1,..,n\}$ and $H_1, .., H_T \subset \{k+1,..,2k\} \times \{2k+1,..,n\}$ for $T\geq 1$, with $|F_t|=|H_t|=m \leq k, \: \forall t \in [T]$. Suppose that $\sum\limits_{j=1}^m \indicatorpar{(i,j) \in F_t} \leq 1, \: \forall i \in [k], t \in [T]$ and similarly
	$\sum\limits_{j=1}^m \indicatorpar{(i,j) \in H_t} \leq 1, \: \forall i \in \{k+1,..,2k\}, t \in [T]$.
	Then, the inequality \eqref{eq:CNF_DNF_inequality} holds. 
	\label{conj:unique_vertex_in_each_clause}
\end{conjecture}

Conjecture \ref{conj:unique_vertex_in_each_clause} was checked numerically and verified for all subsets up to $n \leq 12, m \leq 4, T \leq 3$. 
Three restrictions are imposed on the $F_t, H_t$ sets in the above conjecture: First, the r.v.s $U_i$ are divided into three sets, where variables from the first (second) set can be only the first indices in pairs appearing in $F_t$ ($H_t$), and variables from the third set appear as the second indices in all pairs in $F_t, H_t$. Second, in each $F_t$ or $H_t$ every variable from the first or second set can appear at most once. Finally, the sizes of all sets $F_t,H_t$ are fixed at $m$. All three conditions are necessary. For example, if we allow sets $H_t$ 
of different sizes, a counter-example is given by: 
\begin{align}
	F_1 &= \{(1,5)\}\:; \: F_2=\{(1,6) \}, \nonumber \\
	H_1 &= \{(3,5)\}\:; \: H_2 = \{(3,7), (4,8)\}.
	\label{eq:counter}
\end{align}
Exact enumeration of all consistent partial orders yield: 
\begin{align}
& P\Big(U_1 < U_5 , U_6 | \{ U_3 < U_5 \} \cap \big\{ \{U_3 < U_7\} \cup \{U_4 < U_8\} \big\}\Big) = \frac{861}{2100} = 0.41 > \nonumber \\ 
& P\Big(U_1 < U_5 , U_6 | \{ U_3 < U_5 \} \cap \{ U_3 < U_7 \} \cap \{U_4 < U_8\} \Big) = \frac{336}{840} = 0.4. 
\end{align}

\subsubsection{Computations}
We have developed a computational pipeline for checking whether an inequality of the form \eqref{eq:CNF_DNF_inequality} holds for specific cases, by exhaustive counting of all of the permutations induced by the partial ordering events specified by $F_t, H_s$. When $T=S=1$, the problem reduces to counting linear extensions of a partial order induced by the $E_{ij}$ events. 

For a general partial order, it was shown in \cite{brightwell1991counting} that counting the number of linear extensions is $\#P$-complete. When $T > 1$, exact counting of the number of total orders consistent with the restrictions imposed by the $F_t, H_s$ sets is even harder. For some specific cases the computation might be easier - for example, the number of linear extensions of a partial order given by a chain over $m \leq n$ variables (e.g. $U_1 < U_3 < U_7$ for $m=3 < n=7$), is $\frac{n!}{m!}$ by symmetry.

There are known polynomial randomized approximation algorithms for the general case \cite{dyer1991random}. However, these algorithms only yield an approximate answer, hence to verify whether the inequality \eqref{eq:CNF_DNF_inequality} holds for specific cases, we may need to run them to until the error in estimating event probabilities is smaller than $\frac{1}{n!}$, or alternatively use them as a first filtering step, and turn to exact exhaustive computation when the randomized approach provides evidence for a violation of the inequality. 

The above techniques are implemented in a python code for checking whether the inequality \eqref{eq:CNF_DNF_inequality} holds for a particular set collection $\{F_1,..,F_T ; H_1, .., H_S\}$, and for finding putative counter-example by either sampling such sets collections, or by an exhaustive enumeration over the possible sets $F_t, H_s$ of given sizes. The code can be used for small values of $n$, and is available at \href{https://github.com/orzuk/Pareto}{{\it https://github.com/orzuk/Pareto}}. 
For example, the following command verifies the counter-example \eqref{eq:counter}, \\
{\scriptsize
\texttt{$>>>$ check\_CNF\_DNF\_inequality(9, [[(1, 4)], [(1, 4), (2, 4)]] , [[(1, 3)], [(1, 3), (2, 3)]]) \\
(False, 861, 336, 2100, 840)}
}

The next command was used to verify Conjecture \ref{conj:unique_vertex_in_each_clause} for $n \leq 12, m \leq 4, T \leq 3$ by enumerating all $\sim\!10^6$ possible subset collections up to these sizes, running for $\sim\!35$ hours on a standard laptop PC, and failing to find a counter-example among them: \\
{\scriptsize
\texttt{$>>>$ find\_enumerate\_counter\_example\_DNF\_CNF\_inequality(n = 12, num\_edges = 4, num\_sets = 3) \\
False 
}}

While more efficient algorithms and code are available for counting linear extensions of partial orders, for either the general case or special partial orders (see e.g. \cite{pruesse1994generating,kangas2016counting,kangas2020faster,talvitie2018counting}), they can be applied directly only for the case $T=S=1$. Addressing the more general case requires developing new algorithms for efficient counting or enumeration of the total orders consistent with intersections of unions of pairwise ordering events - such algorithms may be used in in order to provide evidence for proposed putative inequalities or find novel counter-examples for small value of $n$, and may be of independent interest in combinatorial applications. These computational improvements, as well as deriving and proving rigorously generalizations of Theorem \ref{thm:dominate_conditioning_inequality} such as Conjecture \ref{conj:unique_vertex_in_each_clause}, are left for future work.

\section{Some related results}
\label{sec:related results}

\subsection{Asymptotic independence}\label{sec:asymptotic independence}
One possible application of Theorem \ref{thm:dominate_conditioning_inequality} is given by the next theorem:
\begin{theorem}\label{thm:asymptotic independence}
For every $n\geq1$, let \be
 p_n\equiv P(E_{1,n})\ \ , \ \ e_n\equiv P(E_{1,n}\cap E_{2,n})\,.
\ee
\begin{enumerate}
 \item For every $k\geq1$
\be\label{eq:e_k,n exact formula}
 e_n=\sum\limits_{\substack{a,b,c,d\in\mathbb{Z}_+ : \\ a+b+c+d=n-2}}
(-1)^{a+b} \binom{n-2}{a\: b \: c \: d} \frac{(a+b+2c+2)^k - (b+c+1)^k-(a+c+1)^k}{(a+c+1)^k(b+c+1)^k(a+b+c+2)^k}\,. 
\ee

\item If $k>1$, then

\be\label{eq:e_k,n approx}
 e_n\sim p^2_{k,n}\sim \left[\frac{\log^{k-1}(n)}{n (k-1)!}\right]^2\ \ \text{as}\ \ n\to\infty
\ee
and hence 
\be\label{eq:little o rho}
 \rho_{k,n}\equiv\text{Corr}(\indicator{E_{1,n}},\indicator{E_{2,n}})=o\left[\frac{\log^{k-1}(n)}{n}\right]\ \ \text{as} \ \ n\to\infty\,.
\ee
\end{enumerate}
\end{theorem}
\begin{remark}
\normalfont In the current analysis $k$ is a fixed parameter of the model. In particular, it is a parameter of the events $E_{1,n},E_{2,n},\ldots,E_{n,n}$ and correspondingly it appears in the expressions of $p_n$ and $e_n$. At the same time, $k$ remains constant all over the analysis to follow. Consequently, for brevity, it is suppressed in the notations: $E_{i,n}$, $p_n$ and $e_n$.
\end{remark}
The first part of Theorem \ref{thm:asymptotic independence} includes a combinatorial expression for $e_n$ which is, to the best of our knowledge, a new one. Theorem \ref{thm:dominate_conditioning_inequality} is applied in order to provide an elementary proof for the second part of Theorem \ref{thm:asymptotic independence}. 
Importantly, an alternative proof for the second part of Theorem \ref{thm:asymptotic independence} stems from existing results. Specifically, it is well-known (see, \textit{e.g.,} Theorems 3 and 4 in \cite{o1981number}) that
\be 
\label{eq:combinatorial formula}
p_n = \sum_{u=1}^n\binom{n-1}{u-1}\frac{(-1)^{u-1}}{u^k}
\ee
and
\be
p_n\sim\frac{\log^{k-1}(n)}{n (k-1)!}\ \ \text{as}\ \ n\to\infty\,.
\label{eq:p_n_first_order}
\ee
The main result of \cite{bai1998variance} implies that there exists a constant $C_k>0$ such that for every $k>1$
\be
\text{Var}\left(\sum_{i=1}^n\indicator{E_{i,n}}\right)\sim C_k\log^{k-1}(n)\ \ \text{as}\ \ n\to\infty\,.
\ee
Thus, because
\begin{align}
\text{Var} \left(\sum_{i=1}^n\indicator{E_{i,n}}\right) &= \sum_{i=1}^n \text{Var} \left(\indicator{E_{i,n}}\right) + 2 \sum_{i=1}^n \sum_{j=i+1}^n \text{Cov}(\indicator{E_{i,n}}, \indicator{E_{j,n}}) \\ 
&= n \text{Var}(Z_1) + n(n-1) \left\{\mathbb{E} \left(\indicator{E_{1,n}} \indicator{E_{2,n}}\right) - \left[\mathbb{E}\left( \indicator{E_{1,n}}\right)\right]^2 \right\} \nonumber \\
&= n p_n (1-p_n) + n(n-1) (e_n - p_n^2)\,, \nonumber
\end{align}
it is possible to deduce that 
\be
 \lim_{n\to\infty}\frac{e_n}{p_n^2}=1+\lim_{n\to\infty}\frac{\text{Var}(\sum_{i=1}^n\indicator{E_{i,n}})}{(np_n)^2}-\lim_{n\to\infty}\frac{1}{np_n}=1\,.
\ee
\subsection{A limit theorem}\label{sec:limit theorems}
In this part, we consider a sequence of models: For each $n\in\mathbb{N}$ assume that $X_1^n,X_2^n,\ldots,X_{m_n}^n$ are i.i.d. random vectors distributed uniformly in the unit hypercube $[0,1]^k$. For every $n\in\mathbb{N}$ and $i\in[m_n]$ denote the event
\be
 E_{i,n}\equiv\big\{X_i\text{ is a maximum of $\{X_1^n,X_2^n,\ldots,X_{m_n}^n\}$}\big\}\,, 
\ee
and the corresponding indicator $Z_{ni}\equiv \indicator{E_{i,n}}$. Since $k$ remains fixed we keep suppressing it in the notations. An implication of Theorem \ref{thm:asymptotic independence} is addressed in the following theorem:

\begin{theorem}\label{thm:LLN+CLT}
\leavevmode
\begin{enumerate}
\item If $m_n=n$ for every $n\in\mathbb{N}$, then for every $k\in\mathbb{N}$:
\be\label{eq:LLN1}
\frac{1}{np_n}\sum_{i=1}^n Z_i^{k,n}\rightarrow1 \ \ \text{as}\ \ n\to\infty\ \ \text{in} \ \ L^2(P)\,.
\ee

\item Let $(m_n)_{n=1}^\infty$ be a sequence of positive integers such that
$\lim\limits_{n \to \infty} \frac{n}{m_n} = \gamma\in(0,1)$ and assume that $k>1$. For any $n\geq1$ and $i\in [m_n]$ denote
\be
U_{ni}\equiv\frac{Z_{ni}-p_{m_n}}{\sqrt{p_{m_n}(1-p_{m_n})}}\,.
\ee
Then,
\be
\sqrt{n}\left(\frac{1}{n}\sum_{i=1}^{n}U_{ni}-\frac{1}{m_n}\sum_{i=1}^{m_n}U_{ni}\right)\xrightarrow{d}\mathcal{N}(0,1-\gamma)\ \ \text{as } n\to\infty\,.
\label{eq:CLT_differences}
\ee

\item 
For any $k > 1$, there exists a sequence $(m_n)_{n=1}^\infty$ of positive integers such that
\be
n \ll m_n \ll n \log^k n , 
\ee
and: 
\be
\frac{1}{\sqrt{n}}\sum_{i=1}^{n}U_{ni} \xrightarrow{d}\mathcal{N}(0,1)\ \ \text{as } n\to\infty\,.
\label{eq:CLT_partial_sum}
\ee

\end{enumerate}

\end{theorem}

\begin{remark}
\normalfont The set of maxima of a set of vectors is usually called the Pareto-front of that set. Thus, the first part of Theorem \ref{thm:LLN+CLT} is an $L^2(P)$ version of a law of large numbers for the size of the Pareto-front. The second part of Theorem \ref{thm:LLN+CLT} might be interpreted as a CLT for the difference between the normalized size of the Pareto-front and a normalized size of a \textit{partial} Pareto-front. The third part provides a CLT for the size of a certain fraction of the Pareto-front. 
\end{remark}

\begin{remark}\normalfont
The second and third parts of Theorem \ref{thm:LLN+CLT} show a CLT for partial sums of the indicator variables $Z_{ni}=\indicator{E_{i,n}}$. Importantly, they cannot be deduced directly from the CLT given in \cite{bai2005maxima} for the entire sum $\sum\limits_{i=1}^n \indicator{E_{i,n}}$ (with $m_n=n$ for every $n\in\mathbb{N}$). It is also unclear if they can be used in order to deduce the CLT given in \cite{bai2005maxima}. 
\end{remark}

\subsection{Proofs}
\label{sec:proofs}

\subsubsection{Proof of Theorem \ref{thm:asymptotic independence}}
\begin{lemma} \label{lemma:exchangalble_monotone}
Let $W_1,..,W_n$ be exchangeable random variables such that $W_1\sim\text{Ber}(p)$ for some $p\in(0,1)$ and denote $W\equiv\sum\limits_{j=1}^nW_j$. Then, 
\be
\mathbb{E}(W | W_1 = 1) \geq n p. 
\ee
\end{lemma}
\begin{proof}
Since $W_1,W_2,\ldots,W_n$ are exchangeable, then Bayes' theorem with the Cauchy-Schwartz inequality yields that
\begin{align}
\mathbb{E}(W | W_1 = 1) &= \sum_{i=0}^n P(W = i | W_1 = 1) i \nonumber \\
&= \sum_{i=0}^n \frac{P(W=i)}{P( W_1 = 1)} P(W_1 = 1 | W = i) i \nonumber \\
&= \frac{1}{p} \sum_{i=0}^n P(W = i) i^2 \nonumber \\
&= n\frac{\mathbb{E}(W^2)}{\mathbb{E}(W)} \geq \frac{[\mathbb{E}(W)]^2}{\mathbb{E}(W)} = \mathbb{E}(W) = np.
\end{align} \qed
\end{proof}

\subsubsection*{Proof (Theorem \ref{thm:asymptotic independence}):} \begin{enumerate}
 \item For every $i\in[n]$ and $j\in[k]$ let $X_{ij}$ be the $j$-th coordinate of the vector $X_i$. In addition, let $x_1,x_2\in[0,1]^k$ such that for every $i=1,2$, $x_i=(x_{i1},x_{i2},\ldots,x_{ik})$. Then, conditioning on the values of $X_1$ and $X_2$ yields that
\begin{align}
e_n&= \int_{[0,1]^{2k}}\indicatorpar{x_1\not\prec x_2,x_2\not\prec x_1} P \left( \bigcap_{i=3}^n \bigcap_{m=1}^2 \bigcup_{j=1}^k \{ X_{ij} > x_{mj} \}\right)dx_1dx_2 \nonumber \\ 
&= \int_{[0,1]^{2k}}\indicatorpar{x_1\not\prec x_2,x_2\not\prec x_1}\left[P \left( \bigcap_{m=1}^2 \bigcup_{j=1}^k \{ X_{3j} > x_{mj}\} \right)\right]^{n-2}dx_1dx_2\,.
\end{align}
The probabilities that appear inside the integral above are as follows:
\begin{align}
P \left( \bigcap_{m=1}^2 \bigcup_{j=1}^k \{ X_{3j} > x_{mj}\} \right) &= 1 - P \left(\bigcap_{j=1}^k \{X_{3j} < x_{1j}\}\right) - P \left(\bigcap_{j=1}^k \{X_{3j} < x_{2j}\}\right) \nonumber \\ 
&+ P \left(\bigcap_{j=1}^k \big\{X_{3j} < \min\{x_{1j}, x_{2j}\}\big\}\right) \nonumber \\ 
&= 1 - \prod_{j=1}^k x_{1j} - \prod_{j=1}^k x_{2j} + \prod_{j=1}^k \min\{x_{1j}, x_{2j}\}\,.
\end{align}

Thus, using the multinomial theorem deduce that
\begin{align} 
e_n &= \int_{[0,1]^{2k}} \Big(1 - \prod_{j=1}^k \indicatorpar{x_{1j}<x_{2j}}- \prod_{j=1}^k \indicatorpar{x_{1j}>x_{2j}}\Big) \nonumber \\ 
& \Big(1 - \prod_{j=1}^k x_{1j} - \prod_{j=1}^k x_{2j} + \prod_{j=1}^k \min\{x_{1j}, x_{2j}\}\Big)^{n-2} dx_1 dx_2 \nonumber \\
&= \int_{[0,1]^{2k}} \Big(1 - \prod_{j=1}^k \indicatorpar{x_{1j}<x_{2j}}- \prod_{j=1}^k \indicatorpar{x_{1j}>x_{2j}}\Big) \nonumber \\
&\sum\limits_{\substack{a,b,c,d\in\mathbb{Z}_+ : \\ a+b+c+d=n-2}}
(-1)^{a+b} \binom{n-2}{a\: b \: c \: d} \left(\prod_{j=1}^k x_{1j}\right)^a \left(\prod_{j=1}^k x_{2j}\right)^b \left(\prod_{j=1}^k \min\{x_{1j}, x_{2j}\}\right)^c dx_1 dx_2 \nonumber \\
&= \sum\limits_{\substack{a,b,c,d\in\mathbb{Z}_+ : \\ a+b+c+d=n-2}}
(-1)^{a+b} \binom{n-2}{a\: b \: c \: d} \Bigg[ 
\prod_{j=1}^k \int_{0}^1\int_0^1 x_{1j}^a x_{2j}^b \big(\min\{x_{1j}, x_{2j}\}\big)^c dx_{1j} dx_{2j} \nonumber \\
&- \prod_{j=1}^k \int_{0}^1\int_0^{x_{2j}} x_{1j}^a x_{2j}^b \big(\min\{x_{1j}, x_{2j}\}\big)^c dx_{1j} dx_{2j} \nonumber \\
&- \prod_{j=1}^k \int_0^1\int_{0}^{x_{1j}} x_{1j}^a x_{2j}^b \big(\min\{x_{1j}, x_{2j}\}\big)^c dx_{2j} dx_{1j} \Bigg] \nonumber \\
&= \sum\limits_{\substack{a,b,c,d\in\mathbb{Z}_+ : \\ a+b+c+d=n-2}}
(-1)^{a+b} \binom{n-2}{a\: b \: c \: d} \Bigg[ \left(\frac{a+b+2c+2}{(a+c+1)(b+c+1)(a+b+c+2)}\right)^k \nonumber \\
&- \left( \frac{1}{(a+c+1)(a+b+c+2)}\right)^k - \left( \frac{1}{(b+c+1)(a+b+c+2)} \right)^k \Bigg] \nonumber \\
&= \sum\limits_{\substack{a,b,c,d\in\mathbb{Z}_+ : \\ a+b+c+d=n-2}}
(-1)^{a+b} \binom{n-2}{a\: b \: c \: d} \frac{(a+b+2c+2)^k - (b+c+1)^k-(a+c+1)^k}{(a+c+1)^k(b+c+1)^k(a+b+c+2)^k} \ . 
\end{align}

\item The proof of Eq. \eqref{eq:e_k,n approx} will follow from the next squeezing argument. \begin{enumerate}
\item 
Applying Lemma \ref{lemma:exchangalble_monotone} to the indicator random variables $\indicator{E_{1,n}},\indicator{E_{2,n}},\ldots,\indicator{E_{n,n}}$ implies that 
\be
\mathbb{E}\left(\sum_{i=1}^n\indicator{E_{i,n}} \ \big|\ \indicator{E_{1,n}}=1\right) = 1 + (n-1) \frac{e_n}{p_n} \geq n p_n. 
\ee
As a result deduce that 
\be
\frac{e_n}{p^2_{k,n}} \geq \frac{n}{n-1}- \frac{1}{p_n(n-1)}\,. 
\ee
Thus, since $np_n\rightarrow\infty$ as $n\to\infty$, deduce that the lower bound tends to one as $n\to\infty$.
\item 
Next we prove that there exists an upper bound of $\frac{e_n}{p^2_{k,n}}$ which tends to one as $n\to\infty$.
To this end, observe that 
\begin{align}
e_n &= P\left(\bigcap_{i=1}^2 \bigcap_{j\in[n]\setminus\{i\}} \overline{B_{ij}}\right) \nonumber \\
&\leq P\left(\bigcap_{i=1}^2 \bigcap_{j=3}^n \overline{B_{ij}}\right) \nonumber \\
&= P\left(\bigcap_{j=3}^n \overline{B_{1j}}\right) P\left(\bigcap_{j=3}^n \overline{B_{2j}}\ \big|\ \bigcap_{j=3}^n \overline{B_{1j}}\right) \nonumber \\
&\leq p_{k,n-1} P\left(\bigcap_{j=3}^n \overline{B_{2j}}\ \big|\ \bigcap_{j=3}^n B_{j1}\right) \nonumber\\&=p_{k,n-1} \left[1-P\left(\bigcup_{j=3}^n B_{2j} \ |\ \bigcap_{j=3}^n B_{j1}\right)\right]\nonumber \\
&= p_{k,n-1} \left[1+\sum_{u=2}^{n-2} \binom{n-2}{u-1} (-1)^{u-1} P\left(\bigcap_{j=3}^{u+1} B_{2j}\ \big|\ \bigcap_{j=3}^n B_{j1}\right)\right]. 
\label{eq:e_k_n_upperbound}
\end{align}
The second inequality above follows from Theorem \ref{thm:dominate_conditioning_inequality}. The last two lines are due to De-Morgan's laws with an inclusion-exclusion principle.
The next step is to evaluate the probability $P(\bigcap\limits_{j=3}^{u+1} B_{2j} | \bigcap\limits_{j=3}^n B_{j1})$ for $u \geq 2$. The set $\{X_{ij}; i\in [n], j\in [k]\}$ contains i.i.d. random variables with a continuous distribution and hence symmetry considerations yield that
\begin{align}
P\left(\bigcap_{j=3}^{u+1} B_{2j} \ \big|\ \bigcap_{j=3}^n B_{j1}\right) &= 
P\left(\bigcap_{j=3}^{u+1} \bigcap_{l=1}^k \{X_{2l} \leq X_{jl} \} \ \big|\ \bigcap_{j=3}^n \bigcap_{l=1}^k \{X_{jl} \leq X_{1l} \} \right) \nonumber \\
&= \left[P\left(\bigcap_{j=3}^{u+1} \{X_{21} \leq X_{j1} \}\ \big|\ \bigcap_{j=3}^n \{X_{j1} \leq X_{11} \} \right)\right]^k \nonumber \\
&= \left\{\frac{P\left[\left(\bigcap\limits_{j=3}^{u+1} \{X_{21} \leq X_{j1}\}\right)\cap\left(\bigcap\limits_{j=3}^n \{X_{j1} \leq X_{11} \} \right)\right]}{P\left(\bigcap\limits_{j=3}^n \{X_{j1} \leq X_{11} \} \right)}\right\}^k \nonumber \\
&= \left(\frac{\frac{1}{n}\cdot\frac{1}{u}}{\frac{1}{n-1}} \cdot\right)^k\nonumber \\
&= \left(\frac{n-1}{nu}\right)^k\,.
\label{eq:cond_prob_upperbound_evaluation}
\end{align}
Plugging Eq. \eqref{eq:cond_prob_upperbound_evaluation} into Eq. \eqref{eq:e_k_n_upperbound} and using Eq. \eqref{eq:combinatorial formula} with the binomial theorem gives,
\begin{align}
\frac{e_n}{p^2_{k,n}} &\leq \frac{p_{k,n-1}}{p^2_{k,n}} \left[1+ \left(\frac{n-1}{n}\right)^k \sum_{u=2}^{n-2} \binom{n-2}{u-1} (-1)^{u-1} u^{-k} \right] \nonumber \\
&= \frac{p_{k,n-1}}{p^2_{k,n}} \left[ 1 + \left(\frac{n-1}{n}\right)^k (p_{k,n-1}-1) \right] \nonumber \\
&= \left(\frac{n-1}{n}\right)^k \left(\frac{p_{k,n-1}}{p_n}\right)^2 + \frac{1}{p_n}\left[1 - \left(1-\frac{1}{n}\right)^k\right] \frac{p_{k,n-1}}{p_n} \nonumber \\
&=\left(\frac{n-1}{n}\right)^k \left(\frac{p_{k,n-1}}{p_n}\right)^2 +\frac{p_{k,n-1}}{p_n}\cdot \frac{1}{np_n}\sum_{u=1}^k\binom{k}{u}\left(-\frac{1}{n}\right)^{u-1}\,.\nonumber
\end{align}
Finally, $np_n\rightarrow\infty$ and $p_{k,n-1}/p_n\rightarrow1$ as $n\to\infty$ imply that the upper bound above tends to one as $n\to\infty$. 
\end{enumerate}
Once we have established the proof of Eq. \eqref{eq:e_k,n approx}, it is possible to prove Eq. \eqref{eq:little o rho}. To this end, observe that

\be
 \lim_{n\to\infty}\frac{e_n-p_n^2}{\left[\frac{\log^{k-1}(n)}{n (k-1)!}\right]^2}=1-1=0
\ee
and hence
\be
 e_n-p_n^2=o\left(\left[\frac{\log^{k-1}(n)}{n (k-1)!}\right]^2\right)\ \ \text{as}\ \ n\to\infty\,.
\ee
Thus, since 
\be
 \rho_{k,n}=\frac{e_n-p_n^2}{p_n(1-p_n)}\,,
\ee
then
Eq. \eqref{eq:little o rho} follows.\qed
\end{enumerate}

\subsubsection{Proof of Theorem \ref{thm:LLN+CLT}}

\begin{lemma}\label{lemma: LLN}
If $np_{m_n}\rightarrow\infty$ as $n\to\infty$, then: 
 
 \be\label{eq:LLN2}
 \frac{1}{n}\sum_{i=1}^nU_{ni}^2\rightarrow1\ \ \text{as}\ \ n\to\infty\ \ \text{in}\ \ L^2(P)\,.
 \ee
\end{lemma}

\begin{proof}
By definition $U_{ni}^2=\frac{(Z_{ni}-p_{m_n})^2}{p_{m_n}(1-p_{m_n})}$ with $Z_{ni} \sim \text{Ber}(p_{m_n})$ and hence
\begin{align}
EU_{n1}^4 &= E \left[\frac{\left(Z_{n1}-p_{m_n}\right)^4}{p_{m_n}^2\left(1-p_{m_n}\right)^2} \right] \nonumber \\
&= \frac{p_{m_n}\left(1-p_{m_n}\right)^4+\left(1-p_{m_n}\right) p_{m_n}^4} {p_{m_n}^2\left(1-p_{m_n}\right)^2} \nonumber \\
&= \frac{(1-p_{m_n})^2}{p_{m_n}}+\frac{p_{m_n}^2}{1-p_{m_n}} \ .
\end{align}
Thus, the assumption that $\underset{n \to \infty}{\lim}np_{m_n}=\infty$ implies that
\be\label{4th moment1}
\lim_{n \to \infty} \frac{EU^4_{n1}}{n} = \lim_{n \to \infty} \left[ \frac{\left(1-p_{m_n}\right)^2}{np_{m_n}}+\frac{p_{m_n}^2}{n(1-p_{m_n})} \right] = 0 \ . 
\ee
Similarly, 
\begin{align}
EU_{n1}^2U_{n2}^2 &= E \left[\frac{\left(Z_{n1}-p_{m_n}\right)^2\left(Z_{n2}-p_{m_n}\right)^2}{p_{m_n}^2\left(1-p_{m_n}\right)^2} \right] \nonumber \\
&= \frac{E\left[\left(Z_{n1}\left(1-2p_{m_n}\right)+p_{m_n}^2\right)\left(Z_{n2}\left(1-2p_{m_n}\right)+p_{m_n}^2\right)\right]}{p_{m_n}^2\left(1-p_{m_n}\right)^2} \nonumber \\
&= \frac{\left(1-2 p_{m_n}\right)^2 e_{k,m_n} + 2p_{m_n}^3 \left(1-2p_{m_n}\right) + p_{m_n}^4}{p_{m_n}^2\left(1-p_{m_n}\right)^2} \nonumber \\
&= \left(\frac{1-2 p_{m_n}}{1-p_{m_n}}\right)^2 \frac{e_{k,m_n}}{p_{m_n}^2} + \frac{2p_{m_n} - 3p_{m_n}^2}{\left(1-p_{m_n}\right)^2} \ . \end{align}
Using Theorem \ref{thm:asymptotic independence} and recalling that $p_{m_n}\rightarrow0$ as $n\to\infty$, deduce that
\be\label{4th moment2}
\lim_{n \to \infty} EU_{n1}^2U_{n2}^2 = \lim_{n \to \infty} \frac{e_{k,m_n}}{p_{m_n}^2} = 1 \ .
\ee
In addition, standard calculations yield that
\begin{align}
\text{Var}\left(\frac{1}{n}\sum_{i=1}^n U^2_{ni}\right)&=\frac{\text{Var}(U^2_{n1})}{n}+\frac{n-1}{n}\text{Cov}\left(U_{n1}^2,U_{n2}^2\right)\\&=\frac{EU^4_{n1}-1}{n}+\frac{n-1}{n}\left(EU_{n1}^2U_{n2}^2-1\right)\nonumber
\end{align}
and hence Eqs. \eqref{4th moment1}, \eqref{4th moment2} imply that
\be
 \text{Var}\left(\frac{1}{n}\sum_{i=1}^n U^2_{ni}\right)\rightarrow0\ \ \text{as}\ \ n\to\infty\,.
\ee
Finally, the result follows since for each $n\geq1$, $E[\frac{1}{n}\sum\limits_{i=1}^n U^2_{ni}]=1$. 
\qed\newline
\end{proof}

\subsubsection*{Proof (Theorem \ref{thm:LLN+CLT}):} \begin{enumerate}
 \item For $k=1$, the proof is immediate. Assume that $k>1$ and $m_n=n$ for every $n\geq1$. Then by Eq. \eqref{eq:little o rho}, $n\rho_{k,m_n} = n \rho_{k,n} \rightarrow\infty$ as $n\to\infty$. Therefore, Lemma \ref{lemma: LLN} implies that
 
 \be
 \frac{1}{n}\sum_{i=1}^nU_{ni}^2\rightarrow1\ \ \text{as}\ \ n\to\infty\ \ \text{in}\ \ L^2(P)\,. 
 \ee
 Since $Z_{ni}^2=Z_{ni}$ for each $i \in [n]$, then 
 \be
 \frac{1}{n}\sum_{i=1}^nU_{ni}^2=\frac{(1-2p_n)\frac{1}{n}\sum_{i=1}^n Z_{ni}+p_n^2}{p_n(1-p_n)}\rightarrow1\ \ \text{as}\ \ n\to\infty\ \ \text{in}\ \ L^2(P)
 \ee
 and the result follows because $p_n\rightarrow0$ as $n\to\infty$. 
 \item $\left\{U_{ni};n\geq1, i\in [m_n]\right\}$ is a triangular array and by symmetry, for each $n\in \mathbb{N}$ the random variables $U_{n1},U_{n2},\ldots,U_{nm_n}$ are exchangeable. Moreover, by construction $EU_{ni}=0$ and $EU^2_{ni}=1$ for every $n\geq1$, $ i\in [n]$. Thus, in order to complete the proof, it is left to verify that the pre-conditions of Theorem 2 in \cite{weber1980martingale} are satisfied. To this end, recall that $\frac{n}{m_n}\rightarrow\gamma\in(0,1)$ as $n\to\infty$ and observe that 
 \be\label{eq:np_km to infinity}
 \lim_{n\to\infty}np_{m_n}=\lim_{n\to\infty}\frac{n}{m_n}\cdot\frac{\log^{k-1}(m_n)}{(k-1)!}=\infty\,.
 \ee

\begin{enumerate} 
 \item Eq. \eqref{eq:little o rho} yields that
 \be
 EU_{n1}U_{n2}=\rho_{k,m_n}\rightarrow0\ \ \text{as}\ \ n\to\infty\,.
 \ee
 
 \item Since $p_{m_n}\rightarrow0$ as $n\to\infty$, then $p_{m_n} < \frac{1}{2}$ for every sufficiently large $n\geq1$. Therefore, using Eq. \eqref{eq:np_km to infinity}, deduce that for every such $n$ 
\be\label{eq:precondition1}
n^{-\frac{1}{2}}\max_{1\leq i\leq m_n} |U_{ni}|\leq\frac{1-p_{m_n}}{\sqrt{n}\sqrt{p_{m_n}(1-p_{m_n})}}<\frac{1}{\sqrt{np_{m_n}}} 
\underset{n\to\infty}{\longrightarrow} 0\,. 
\ee 

\item Due to Eq. \eqref{eq:np_km to infinity}, it is possible to apply Lemma \ref{lemma: LLN} in order to deduce that 

\be
\frac{1}{n}\sum_{i=1}^n U_{ni}^2\rightarrow1\ \ \text{as}\ \ n\to\infty\ \ \text{in}\ \ L^2(P)\,.
\ee
This convergence also yields convergence in probability and the last pre-condition of Theorem 2 in \cite{weber1980martingale} follows. 
\end{enumerate}

\item 
Since $\rho_{k,n} = o\left(\frac{\log^{k-1} (n)}{n}\right)$, there exists a positive decreasing function $g(n)$ such that:
\begin{description}
 \item[(i)] $|\rho_{k,n}| \leq \frac{\log^{k-1} (n)}{n} g(n)$ for every $n\in\mathbb{N}$ large enough.
 
 \item[(ii)] $g(n)\downarrow0$ as $n\to\infty$.
\end{description}
(for example, take $g(n) \equiv \underset{n' \geq n}{\sup}\frac{|\rho_{k,n'}| n' }{\log^{k-1} (n')}$). 
Then, for each $n\in\mathbb{N}$, define
\be
 m_n\equiv\min\left\{m\in\mathbb{N}\ ;\ n < \frac{m}{\log^{k-1} (m)\max\big( \sqrt{g(m)}, 1 / \sqrt{\log m} \big) } \right\} 
\label{eq:n_vs_m_n}
\ee
and note that $m_n<\infty$. The proof will follow by verifying that the following four pre-conditions of Corollary 1 in \cite{weber1980martingale} are satisfied. 
\begin{enumerate}
\item First, \eqref{eq:n_vs_m_n} and the definition of $g(n)$ imply that for $n$ large enough,
\begin{align}
n |\rho_{k, m_n}| &\leq \frac{m_n}{\max\big( \sqrt{g(m_n)}, 1 / \sqrt{\log (m_n)} \big) \log^{k-1} (m_n) } \cdot \frac{\log^{k-1} (m_n)}{m_n} g(m_n) \nonumber \\
&= \min\big( \sqrt{g(m_n)}, g(m_n) \sqrt{\log (m_n)} \big) \underset{n\to\infty}{\longrightarrow} 0.
\end{align}

\item Next, the definition of $m_n$ yields that for every $n\geq1$,
\begin{equation}
 n \geq \frac{m_n-1}{\log^{k-1} (m_n-1)\max\big( \sqrt{g(m_n-1)}, 1 / \sqrt{\log (m_n-1)} \big) }\,.
\end{equation} 
In addition, the first order approximation \eqref{eq:p_n_first_order} implies that $p_{m_n} > \frac{1}{2} \frac{\log^{k-1} (m_n)}{(k-1)!m_n}$ for every sufficiently large $n$. Combining the two inequalities gives as $n\to\infty$:  
\begin{align}
n p_{m_n} &> \frac{m_n-1}{\log^{k-1} (m_n-1)\max\big( \sqrt{g(m_n-1)}, 1 / \sqrt{\log (m_n-1)} \big) } \cdot \frac{\log^{k-1} (m_n)}{2(k-1)!m_n} \nonumber \\
&\sim \frac{1}{2(k-1)! \max\big( \sqrt{g(m_n-1)}, 1 / \sqrt{\log (m_n-1)} \big) } \underset{n\to\infty}{\longrightarrow} \infty.
\label{eq:n_p_k_m_n_to_infty}
\end{align}

Therefore, in similar to Eq. \eqref{eq:precondition1}, 
\be
n^{-\frac{1}{2}}\max_{1\leq i\leq m_n} |U_{ni}| < \frac{1}{\sqrt{n p_{m_n}}} \underset{n\to\infty}{\longrightarrow} 0 . 
\ee

\item By Eq. \eqref{eq:n_p_k_m_n_to_infty}, $n p_{m_n} \underset{n\to\infty}{\longrightarrow} \infty$. Hence, Lemma \ref{lemma: LLN} yields that 
 
\be\label{eq:LLN3}
\frac{1}{n}\sum_{i=1}^nU_{ni}^2\rightarrow1\ \ \text{as}\ \ n\to\infty\ \ \text{in}\ \ L^2(P)
\ee
and this result implies convergence in probability. 

\item Finally, 
\be
\frac{n}{m_n} < \min\left[ \frac{1}{\sqrt{g(m_n)}\log^{k-1} (m_n)}\ ,\ \frac{1}{\log^{k-3/2} (m_n)}\right] 
\underset{n\to\infty}{\longrightarrow} 0.
\label{eq:n_m_n_ratio_to_zero}
\ee

\end{enumerate}
Since the above four conditions are satisfied, by Corollary 1 in \cite{weber1980martingale}, the Gaussian limit in Eq. \eqref{eq:CLT_partial_sum} holds. 
By condition \eqref{eq:n_m_n_ratio_to_zero} above, $n \ll m_n$. 
Because $g(m)<1$ for $m$ large enough, we get from Eq. \eqref{eq:n_vs_m_n} that $n > c \frac{m_n}{\log^{k-1} (m_n)}$ for some $c>0$. Therefore, for $m_n$ large enough,
\be
n \log^{k-1} n > c \frac{m_n}{\log^{k-1} (m_n)} \Big[\log c + \log m_n - \log\big(\log^{k-1}(m_n)\big)\Big]^{k-1} > c \frac{m_n \big(\frac{1}{2}\log (m_n)\big)^{k-1}}{\log^{k-1} (m_n)} = \frac{c}{2^{k-1}} m_n
\ee

which gives $m_n < c' n \log^{k-1} (n)$ for some $c'>0$ and hence $m_n \ll n \log^{k}(n)$, completing the proof of the theorem. 
 \qed
\end{enumerate}

\bibliographystyle{unsrt}
\bibliography{XYZ_DNF_CNF}

\end{document}